\documentclass[11pt, reqno]{amsart}
\numberwithin{equation}{section}
\usepackage{dsfont}
\usepackage{amscd,amsmath,amssymb,amsfonts}
\usepackage{tikz,color}
\usepackage{hyperref}
\usepackage{enumerate}
\usepackage{latexsym}
\usepackage{cases}
\usepackage{amsthm}
\usepackage{graphicx}
\usepackage{verbatim}
\usepackage{mathrsfs}
\usepackage{geometry}
\newcommand{\R}{\mathbb{R}}

\newtheorem{theorem}{Theorem}[section]

\newtheorem{lemma}[theorem]{Lemma}

\newtheorem{conjecture}[theorem]{Conjecture}
\newtheorem{proposition}[theorem]{Proposition}

\theoremstyle{definition}

\newtheorem{remark}[theorem]{Remark}

\makeatletter
\newcommand{\Extend}[5]{\ext@arrow0099{\arrowfill@#1#2#3}{#4}{#5}}
\makeatother

\begin{document}
\title[Restriction estimates]{A Restriction estimate for a certain surface of finite type in $\mathbb{R}^3$}

%
%

\author{Zhuoran Li}
\address{The Graduate School of China Academy of Engineering Physics, P. O. Box 2101,\ Beijing,\ China,\ 100088 ;}
\email{lizhuoran18@gscaep.ac.cn}

\author{Changxing Miao}
\address{Institute of Applied Physics and Computational Mathematics, Beijing 100088}
\email{miao\_changxing@iapcm.ac.cn}

\author{Jiqiang Zheng}
\address{Institute of Applied Physics and Computational Mathematics, Beijing 100088}
\email{zhengjiqiang@gmail.com and zheng\_jiqiang@iapcm.ac.cn}

\begin{abstract}
In this paper,  we study the restriction estimate for a certain surface of finite type in $\R^3$, and partially improves the results of Buschenhenke-M\"{u}ller-Vargas \cite{BMV}. The key ingredients of the proof include the so called generalized rescaling technique based on a decomposition adapted to finite type geometry,  a decoupling inequality and  reduction of  dimension arguments.

\end{abstract}

\maketitle

\begin{center}
 \begin{minipage}{120mm}
 \end{minipage}
 \end{center}





\section{Introduction  and  main result}

\noindent

Broadly speaking, the theme of restriction problems is to study the properties of Fourier transforms of measures supported on curved manifolds. It is intimately connected with geometric measure theory and incidence geometry, and has led to important developments in dispersive partial differential equations and number theory.

Let $S$ be a given smooth compact hypersurface in $\mathbb{R}^n$, and $d\sigma$ is its induced Lebesgue measure. The Fourier restriction problem for $S$, proposed by E. Stein in an unpublished work dating back to the late 1960s(see also \cite{Stein}), asks for the range of exponents $p$ and $q$ for which the inequality
\begin{equation}\label{equ:restriction}
  \Big( \int_{S}\vert \widehat{f} \vert^q d\sigma \Big)^{\frac1q} \leq C \Vert f \Vert_{L^p(\mathbb{R}^n)}
\end{equation}
holds true for each $f\in \mathcal{S}(\mathbb{R}^n)$, with a constant $C=C(p,q,S)$ independent of $f$.

There had been a vast amount of work on this problem by many mathematicians in the 1970s and 1980s.
For $n=2$, the sharp estimate for curves  with non-vanishing
curvature had been solved  by C. Fefferman, E. Stein and A. Zygmund \cite{F70,Z74}. In higher dimensions, the sharp $L^p$-$L^2$ estimate  for hypersurfaces with non-vanishing Gaussian curvature was obtained by E.Stein and P.Tomas \cite{Stein86, Tomas}, see also R. Strichartz \cite{Str77}.

 The problem on general $L^p$-$L^q$ restriction estimates appears to be an extremely open problem.
 The important progress has been made by many pre-eminent mathematicians in 1990s,  and major new ideas were introduced by J. Bourgain \cite{Bo91,Bo95}, which led to many modern methods such as wave packet decomposition, induction on scales etc.
 In practice,  we  usually write \eqref{equ:restriction} in its adjoint form:
\begin{equation*}
  \| \mathcal E_{S}g \|_{L^{p'}(\mathbb{R}^n)}\leq C\|g\|_{L^{q'}(S, d\sigma)},\;\;\; p'=\tfrac{p}{p-1},\; q'=\tfrac{q}{q-1},
\end{equation*}
where
\[\mathcal E_{S}g(x):=\int_{S}g(\omega)e^{2\pi i x\cdot \omega}d\sigma,\]
denotes the standard Fourier extension operator associated with the hypersurface $S$. For the surfaces with positive definite second fundamental
form in $\mathbb{R}^3$ such as paraboloid and sphere, Stein's conjecture can be stated as
\begin{equation}\label{equ:parglo}
  \| \mathcal E_{S}g \|_{L^{p}(\mathbb{R}^3)}\leq C\|g\|_{L^{\infty}(S,d\sigma)},
\end{equation}
for all $p>3$. By the standard $\varepsilon$-removal argument in \cite{BG, Tao99}, \eqref{equ:parglo} reduces to a local version which would be stated as:
\begin{conjecture}[Local version on restriction conjecture in 3D]\label{con:local3d}
Let $S$ be a smooth compact surface in $\R^3$ with positive definite second fundamental form.
For any given $\varepsilon >0$,  there holds
\begin{equation}\label{equ:parloc}
  \| \mathcal E_{S}g \|_{L^{p}(B_{R})}\leq C(\varepsilon)R^{\varepsilon}\|g\|_{L^{\infty}(S,d\sigma)}, \;\;\;
  R\ge1,
\end{equation}
for all $p >3$. Here $B_R$ denotes a ball centred at the origin with radius $R$ in $\R^3$.
\end{conjecture}

 Tao \cite{Tao} obtained the estimate \eqref{equ:parloc} for $p>3+\frac{1}{3}$ by making use of bilinear methods, which was introduced by Wolff \cite{Wolff}. Bourgain-Guth \cite{BG} proved \eqref{equ:parloc} for $p>\frac{56}{17}\approx 3.29$ by developing so-called Bourgain-Guth argument based on multilinear restriction theory established by Bennett, Carbery and Tao \cite{BCT}. Later on, Guth \cite{Guth} improved the range of $p$ to $p>3.25$  using polynomial partitioning, which combines Dvir's idea in the work \cite{D} on Kakeya problem in finite field with the argument of Guth-Katz \cite{GK} on incidence geometry. Recently, using brooms in the polynomial partitioning setting,  Wang \cite{Wang} refreshed  the range of integral exponents on restriction estimate \eqref{equ:parloc} to  $p>3+\frac{3}{13}$.

How about the restriction problem for general hypersurfaces without assumption on the curvature?
 Stein \cite{Stein93} first proposed the restriction problem for the surface of finite type in $\R^n$. In this direction, some partial results have been obtained by I. Ikromov, M. Kempe and D. M\"{u}ller \cite{IKM10, IM11, IM15}, and the sharp range of Stein-Tomas type restriction estimate has been determined for a large class of smooth finite type hypersurfaces, including all analytic hypersurfaces. For the pointwise  estimate of a class of oscillatory integrals related to dispersive equations under the geometric assumptions of finite type, we  refer to Chen-Miao-Yao\cite{CMY}.

 Our purpose in this work is to consider a certain model class of surfaces in $\R^3$ with varying curvature and study the associated  adjoint restriction estimates.
Consider the surfaces in $\mathbb{R}^3$ given by
\begin{equation*}
  F^2_m:=\Big\{(\xi_1,\xi_2,\xi_1^m+\xi_2^m): (\xi_1,\xi_2) \in [0,1]^2\Big\},
\end{equation*}
where $m\ge 2$ is an integer. For each subset $Q\subset [0,1]^{2}$,  we denote the  Fourier extension operator  on $F^2_m$ by
\begin{equation*}
  \mathcal E_{Q}g(x):= \int_Q g(\xi_1,\xi_2)e\big(x_1\xi_1+x_2\xi_2+x_3(\xi^m_1+\xi_2^m)\big)d\xi
\end{equation*}
where $e(t)=e^{2\pi i t}$. Our goal is to prove that
\begin{equation}\label{equ:goal}
  \|\mathcal E_{[0,1]^2}g\|_{L^{p}(B_R)}\leq C(\varepsilon)R^{\varepsilon}\|g\|_{L^{\infty}(F^2_m)}
\end{equation}
for certain $p$. In this direction, Buschenhenke, M\"{u}ller and Vargas first  considered the $L^{r}$-$L^{p}$ restriction problem for a class
  of finite type surfaces in \cite{BMV} based on bilinear methods, see Remark \ref{rem:BMV} below for a discussion.

For $m=2$, the surface $F^2_2$ is exactly paraboloid over the region $[0,1]^2$. Since paraboloid enjoys nice geometry properties such as all of its principle curvatures are nearly equal to one, we can run the standard parabolic rescaling techniques successfully in procession of induction on scales.  For more detailed analysis, one can refer to Subsection 2.2.

For $m>2$, the Gaussian curvature of the surface $F_m^2$ vanishes when $\xi_1=0$ or $\xi_2=0$. From now on, we focus on the surface $F^2_4$, which  is degenerate on certain submanifolds with codimension two. However, to our surprise, such a degenerate surface might possess the same conjectured range of $p$ as in the paraboloid case such that $L^{\infty}$-$L^{p}$ restriction estimate holds in the setting of extension operator. We remark that $L^{r}$-$L^3$ restriction estimate holds true for the surface $F^2_4$ only if $r=\infty$, which differs from the paraboloid case.

In the setting of $L^r$-$L^p$ estimate, the necessary conditions on restriction inequality
for the surface $F^2_4$ are stated as following:
\begin{equation*}
\|\mathcal E_{[0,1]^2}g\|_{L^{p}(B_R)}\leq C(\varepsilon)R^{\varepsilon}\|g\|_{L^{r}(F^2_4)}
\end{equation*}
holds only if $p>\max\{3, 3r'\}$ or
\begin{equation*}
        p>3\;\;\text{and}\;\; p\ge\max\big\{r, 3r'\big\}.
\end{equation*}
For more systematic discussion on the necessary conditions, one can refer to the work in \cite{BMV,SS}.
\vskip0.2cm

Why do we focus on the case $m=4$?  We believe that it corresponds to interesting case which
is much different with $m\ge 6$. Firstly we make scaling analysis to see what happens.
Let $K$ be a large number with $1\ll K \ll R^\varepsilon$. Under the assumption that \eqref{equ:goal} holds, we have
\begin{equation*}
\|\mathcal E_{[0,1]^2}g\|_{L^p(B_{RK^{-\frac1m}})}\leq C(\varepsilon)\Big(\frac{R}{K^{\frac1m}}\Big)^{\varepsilon}\|g\|_{L^{\infty}(F^2_m)},
\end{equation*}
We pick out a small but typical region:
\begin{equation*}
  \Omega:=[0,K^{-\frac{1}{m}}]\times[0,K^{-\frac{1}{m}}].
\end{equation*}
  This piece of surface  defined on region $\Omega$ is similar to the whole one, but with a smaller scale. By the change of variables: $\xi=K^{-\frac1m}\eta$, we have
\begin{align*}
   \mathcal E_{\Omega}g(x):=& \int_{\Omega} g\big(\xi_1,\xi_2\big)e\Big(x_1\xi_1+x_2\xi_2+x_3\sum^{2}_{i=1}\xi^m_i\Big)d\xi\\
   =&\int_{[0,1]^2} \tilde{g}(\eta_1,\eta_2)e\Big(K^{-\frac{1}{m}}x_1\eta_1+K^{-\frac{1}{m}}x_2\eta_2+K^{-1}x_3\sum^{2}_{i=1}\eta^m_i\Big)d\eta\\
=& \big(\mathcal E_{[0,1]^2}\tilde{g}\big)(\tilde{x}),
\end{align*}
where
\begin{align*}
  \tilde{g}(\eta_1,\eta_2):=&K^{-\frac{2}{m}}g(K^{-\frac{1}{m}}\eta_1,K^{-\frac{1}{m}}\eta_2),
\end{align*}
and
\begin{equation*}
  \tilde{x}=(K^{-\frac{1}{m}}x_1,K^{-\frac{1}{m}}x_2,K^{-1}x_3).
\end{equation*}
Hence, we derive that
\begin{align*}
\|\mathcal E_{\Omega}{g}\|_{L^p(B_R)}\leq& K^{\frac{m+2}{pm}}\|\mathcal E_{[0,1]^2}\tilde{g}\|_{L^p(B_{\frac{R}{K^{1/m}}})}\\
\leq& C(\varepsilon)R^{\varepsilon}K^{\frac{m+2}{pm}-\frac{2}{m}-{\frac{1}{m}}\varepsilon}\|g\|_{L^{\infty}(F^2_m)}.
\end{align*}
Thus, it is reasonable to require $\frac{m+2}{pm}-\frac{2}{m}\leq0$, i.e. $p\geq\frac{m+2}{2}$ to
obtain the control estimate of $\mathcal E_{\Omega}{g}$.
If we want to prove \eqref{equ:goal} for $p\ge 3$, the same range as the conjecture of the paraboloid case, we should work under the assumption that $\frac{m+2}{2}\leq 3$, i.e. $m\leq 4$.
\vskip0.15cm

Now we state the main result.

\begin{theorem}\label{thm:main}
Let $p>3.3$. There holds
\begin{equation}\label{equ:goalm}
  \|\mathcal E_{[0,1]^2}g\|_{L^{p}(B_R)}\leq C(\varepsilon,p)R^{\varepsilon}\|g\|_{L^{\infty}(F^2_4)}.
\end{equation}

\end{theorem}

\begin{remark}\label{rem:BMV}
    Buschenhenke-M\"{u}ller-Vargas in \cite{BMV} considered the $L^{r}$-$L^{p}$ restriction problem for a class
  of finite type surfaces and proved that \eqref{equ:goalm} holds true  for all $p>\frac{10}{3}$ by bilinear methods. Recently, Schwend and Stovall studied such problems in all dimensions in \cite{SS}. Their results also imply that \eqref{equ:goalm} holds  for all $p>\frac{10}{3}$ in $\R^3$.
  \end{remark}

{\bf Outline of the proof of Theorem \ref{thm:main}:}
 The surface $F^2_4$ is badly behaved when $\xi_1=0$ or $\xi_2=0$.  Roughly speaking, the strategy is to single out small neighborhoods of these two lines where we apply reduction of dimension arguments. More precisely, we first divide $[0,1]^2$ into $\bigcup\limits_{j=0}^3\Omega_j$ as Fig.1 below
 \begin{center}
 \begin{tikzpicture}[scale=0.6]

\draw[->] (-0.2,0) -- (5,0) node[anchor=north]   {$\xi_1$};
\draw[->] (0,-0.2) -- (0,5)  node[anchor=east] {$\xi_2$};

 \draw[thick] (0,4) -- (4,4)
 (4,0)--(4,4)
 (0,1)--(4,1)
 (1,0)--(1,4) ;

 \draw (-0.15,0) node[anchor=north] {O}
      (1.2,-0.08) node[anchor=north] {$K^{-\frac14}$}
     (4,-0.08) node[anchor=north] {$1$}
    (2.5,2.5) node[anchor=north] {$\Omega_0$}
     (2.5,0.8) node[anchor=north] {$\Omega_1$}
     (0.5,2.5) node[anchor=north] {$\Omega_2$}
      (0.5,0.8) node[anchor=north] {$\Omega_3$};

 \draw  (-0.15, 1) node[anchor=east] {$K^{-\frac14}$}
       (0,4) node[anchor=east] {$1$};

 \path (2,-1.5) node(caption){Fig. 1};  
 \end{tikzpicture}
\end{center}
 where $K$ is a large number independent of $R$  with  $1\ll K \ll R^\varepsilon$.
 Then, we have
 \begin{align}\label{equ:e01intr}
    \|\mathcal E_{[0,1]^2}g\|_{L^p(B_R)}\leq & \sum_{j=0}^{3} \|\mathcal E_{\Omega_j}g\|_{L^p(B_R)}.
 \end{align}
Denote $Q_p(R)$ to be  the least number such that
\begin{equation}\label{equ:defqpr1}
  \|\mathcal E_{[0,1]^2}g\|_{L^p(B_R)}\leq Q_p(R)\|g\|_{L^{\infty}(F_4^2)}.
\end{equation}

We first consider the contribution of $\Omega_0$-part. By Guth\cite{Guth}, we obtain
\begin{equation}\label{equ:guth325}
  \|\mathcal E_{\Omega_0}g\|_{L^p(B_R)}\leq C(K)C(\varepsilon)R^{\varepsilon}\|g\|_{L^{\infty}(F_4^2)},  \;\; \forall\; p>3.25,
\end{equation}
where $C(K)$ is a fixed positive power of $K$ and $C(\varepsilon)$ is the constant appeared in restriction estimate for perturbed paraboloid.
For the contribution of  $\Omega_3$-part, we get by rescaling and the definition of $Q_p(\cdot)$ in \eqref{equ:defqpr1}
\begin{equation}\label{equ:omeest1}
  \|\mathcal E_{\Omega_3}g\|_{L^p(B_R)}\leq CK^{\frac{3}{2p}-\frac{1}{2}}Q_p\Big(\tfrac{R}{K^{\frac14}}\Big)\|g\|_{L^{\infty}(F_4^2)}.
\end{equation}

 The most difficult part is to estimate the contribution of $\Omega_1$-part and $\Omega_2$-part.  By symmetry, it suffices to treat $\Omega_1$-part. For the subregion $[K^{-1/4},1]\times [0,R^{-1/4}]\subset \Omega_1$, we will introduce the wave packet decomposition and establish the related square function estimate, which are crucial in reduction of dimension arguments. Combining these with an argument based on 2D Kakeya-type estimate, we derive that
\begin{equation}\label{equ:omeg1partcon0}
   \|\mathcal E_{[K^{-1/4},1]\times [0,R^{-1/4}]}g\|_{L^p(B_R)}\leq C_{\varepsilon}R^{\varepsilon}\|g\|_{L^{\infty}(F_4^2)},\;\;p>3.
\end{equation}
For the subregion $\Omega_1\setminus [K^{-1/4},1]\times [0,R^{-1/4}]$, we will establish a certain decoupling inequality and use it to deduce that
\begin{equation}\label{equ:omeg1partcon1}
   \big\|\mathcal E_{\Omega_1\setminus [K^{-1/4},1]\times [0,R^{-1/4}]}g\big\|_{L^p(B_R)}\leq  C(K)C_{\varepsilon}R^{\varepsilon}\|g\|_{L^{\infty}(F_4^2)},\;p>3.3,
\end{equation}
where the $C(K)$ is a fixed positive power of $K$.
 It follows that
\begin{equation}\label{equ:omeg1partcon}
   \|\mathcal E_{\Omega_1}g\|_{L^p(B_R)}\leq  C(K)C_{\varepsilon}R^{\varepsilon}\|g\|_{L^{\infty}(F_4^2)},\;p>3.3.
\end{equation}
Collecting the contributions from different parts, we obtain for $p>3.3$
$$\|\mathcal E_{[0,1]^2}g\|_{L^p(B_R)}\leq \Big[C(K)C(\varepsilon)R^{\varepsilon}+2C(K)C_{\varepsilon}R^{\varepsilon}
+CK^{\frac{3}{2p}-\frac{1}{2}}Q_p\Big(\tfrac{R}{K^{\frac14}}\Big)\Big]\|g\|_{L^{\infty}(F_4^2)}.$$
This inequality implies the following recurrence inequality
$$Q_p(R)\leq C(K)C(\varepsilon)R^{\varepsilon}+2C(K)C_{\varepsilon}R^{\varepsilon}
+CK^{\frac{3}{2p}-\frac{1}{2}}Q_p\Big(\tfrac{R}{K^{\frac14}}\Big),$$
where the $C(K)$ is a fixed positive power of $K$.
Note that when $p>3.3$,  the exponent $\frac{3}{2p}-\frac{1}{2}$ is negative. Therefore, for appropriate choice of $K$ one may iterate this inequality to deduce $Q_p(R)\lesssim_{\varepsilon}R^{\varepsilon}$. This will help us to finish the proof of Theorem \ref{thm:main}.


\vskip 0.2in

The paper is organized as follows.  In Section 2, we will give some preliminaries on the problem and introduce a quite powerful decomposition associated with finite type surface which plays an important role in the proof of Theorem \ref{thm:main}. In Section 3, we give the proof of Theorem \ref{thm:main} by reduction of dimension arguments including a decoupling inequality, a square function estimate and a Kakeya-type estimate.

\vskip 0.2in

{\bf Notations:} For nonnegative quantities $X$ and $Y$, we will write $X\lesssim Y$ to denote the estimate $X\leq C Y$ for some
large constant C which may vary from line to line and depend on various parameters,
and similarly we use $X\ll Y$ to denote $X\leq C^{-1}Y$. If $X\lesssim Y\lesssim X$, we simply write $X\sim Y$. Dependence of implicit constants on the power $p$ or the dimension will be suppressed; dependence on additional parameters will be indicated by subscripts. For example, $X\lesssim_u Y$ indicates $X\leq CY$ for some $C=C(u)$. Usually,  Fourier transform on $\mathbb{R}^n$
is defined by
\begin{equation*}
\aligned \widehat{f}(\xi):= \int_{\mathbb{R}^n}e^{- 2\pi ix\cdot \xi}f(x)\,dx.
\endaligned
\end{equation*}




\section{Preliminaries and useful decomposition}

To analyse our extension operator $\mathcal E$, we need to partition the surface $F^2_4$ into small pieces in an appropriate manner. For this purpose, we present a few concepts and relative calculations on the differential geometry of the surface.  This helps us to introduce a decomposition adapted to finite type surface in Subsection 2.2, which plays a central role in the proof of Theorem \ref{thm:main}.

\subsection{Gaussian curvature of the surface $F^2_4$}\label{subsec:difgeo}

 In this subsection, we recall some useful preliminaries on geometry of the surface $\gamma(u,v)=\{(u,v,h(u,v)):(u,v)\in [0,1]^2\}$ in $\mathbb{R}^3$, where $h(u,v)=u^4+v^4$. It is easy to see that this surface has some degenerate features by calculating its Gaussian curvature, and one can refer to standard textbooks such as \cite{Do}. We calculate the tangent vectors and the unit norm vector at $(u,v)$ as follows
\begin{align*}
\gamma_u=&(1,0,4u^3),\\
\gamma_v=&(0,1,4v^3),\\
n=&\frac{\gamma_u \times \gamma_v}{\vert \gamma_u \times \gamma_v \vert}.
\end{align*}
Also,  a direct calculation yields
\begin{align*}
  \gamma_{uu}=&(0,0,12u^2),\\
  \gamma_{uv}=&(0,0,0),\\
  \gamma_{vv}=&(0,0,12v^2).
\end{align*}
Therefore, the first fundamental form of the surface $\gamma(u,v)$ is given by
\begin{equation*}
  {\rm I}(du,dv)= Edu^2+2Fdudv+Gdv^2,
\end{equation*}
where
\begin{align*}
  E=&\langle \gamma_u,\gamma_u \rangle= 1+16u^6,\\
  F=&\langle \gamma_u,\gamma_v \rangle=16u^3v^3,\\
  G=&\langle \gamma_v,\gamma_v \rangle=1+16v^6.
\end{align*}
And the second fundamental form is given by
\begin{equation*}
 {\rm II}(du,dv)= Ldu^2+2Mdudv+Ndv^2,
\end{equation*}
where
\begin{align*}
L=&\langle \gamma_{uu},n \rangle=\frac{12u^2}{\sqrt{1+16u^6+16v^6}},\\
M=&\langle \gamma_{uv},n \rangle=0,\\
N=&\langle \gamma_{vv},n \rangle=\frac{12v^2}{\sqrt{1+16u^6+16v^6}}.
\end{align*}
Hence, the Gaussian curvature of the surface $\gamma(u,v)$ is
\begin{equation}\label{equ:gausscur}
  k=\frac{LN-M^2}{EG-F^2}=\frac{144u^2v^2}{(1+16u^6+16v^6)^2}.
\end{equation}

\begin{remark}\label{rem:geocal}
From the above calculation, we see that the Gaussian curvature of $\gamma(u,v)$ vanishes when $u=0$ or $v=0$.
We observe that the surface has positive definite second fundamental form if both $u$ and $v$ are away from zero.
In this region, one can adopt the known restriction estimates for perturbed paraboloid, see Guth \cite{Guth}. So it reduces to the case $0 < u \ll 1$ or $0 < v \ll 1$. In other words, we only need to consider the surface in a small neighborhood of the curve $(0,v,v^4)$ or $(u,0,u^4)$. We will apply reduction of dimension arguments to these small neighborhoods.
\end{remark}

\subsection{A decomposition adapted to finite type surface}

Now we introduce a useful decomposition adapted to the surface $F^2_4$, which enables us to do the so-called generalized rescaling successfully in process of induction on scales.
To begin the discussion, we first recall the prototypical case, the standard parabolic rescaling,  for the surface $F^2_2$ as follows.
Let $K$ be a large number satisfying $1\ll K \ll R^{\varepsilon}$ for any fixed $\varepsilon >0$. Suppose that $\lambda$ and $\sigma$ are two given numbers between $0$ and $1$. For each subset $Q\subset [0,1]^{2}$ and $\phi(\xi_1,\xi_2)=\xi^2_1+\xi^2_2$, we denote the Fourier extension operator associated with the paraboloid by
\begin{equation*}
 {\mathcal  E}^{\rm Par}_{Q}g(x):= \int_Q g(\xi_1,\xi_2)e[x_1\xi_1+x_2\xi_2+x_3\phi(\xi_1,\xi_2)]\;d\xi_1\;d\xi_2.
\end{equation*}
For the region $\tau:= [\lambda,\lambda+K^{-\frac12}]\times[\sigma,\sigma+K^{-\frac12}]$, we take change of variables
\begin{equation*}
  \left\{\begin{aligned}
    &\xi_1=\lambda+\tfrac{\eta_1}{K^{\frac12}},\\
    &\xi_2=\sigma+\tfrac{\eta_2}{K^{\frac12}}
  \end{aligned}\right.
\end{equation*}
to deduce
\begin{equation*}
  \vert \mathcal{E}^{\rm Par}_{\tau}g(x)\vert=\Big| \int_{[0,1]^2}\tilde{g}(\eta_1,\eta_2)e[\tilde{x}_1\eta_1+\tilde{x}_2\eta_2+\tilde{x}_3\phi(\eta_1,\eta_2)]
  \;d\eta_1\;d\eta_2\Big|\\
  =\vert \mathcal{E}^{\rm Par}_{[0,1]^2}\tilde{g}(\tilde{x})\vert.
\end{equation*}
Here
\begin{align*}
&\tilde{x}=\Big(K^{-\frac12}(x_1+2\lambda x_3),K^{-\frac12}(x_2+2\sigma x_3),K^{-1}x_3\Big),\\
&\tilde{g}(\eta_1,\eta_2)=K^{-1}g\Big(\lambda+\tfrac{\eta_1}{K^{\frac12}},\sigma+\tfrac{\eta_2}{K^{\frac12}}\Big).\\
\end{align*}
See Fig.2 for the  process of  parabolic rescaling.

\begin{center}
 \begin{tikzpicture}[scale=0.7]

\draw[->] (-0.2,0) -- (4,0) node[anchor=north] {$\xi_1$};
\draw[->] (0,-0.2) -- (0,4)  node[anchor=east] {$\xi_2$};

\draw (-0.15,0) node[anchor=north] {O}
(3.2,0) node[anchor=north] {$1$}
(1,0) node[anchor=north] {$\lambda$}
(2.1,0) node[anchor=north] {\small $\lambda+{K^{-\frac12}}$};

\draw (0,3.2) node[anchor=east] {1}
      (0,1) node[anchor=east] {$\sigma$}
       (0,1.75) node[anchor=east] {\small $\sigma+{K^{-\frac12}}$};

\draw[thick] (0,3.2) -- (3.2,3.2)
             (3.2,0) -- (3.2,3.2)
             (1,1) -- (1,1.75)
             (1,1) -- (1.75,1)
             (1,1.75) -- (1.75,1.75)
             (1.75,1) -- (1.75,1.75);

\draw[dashed,thick] (1,0) -- (1,1) 
                     (1.75,0) -- (1.75,1)
                     (0,1) -- (1,1)
                     (0,1.75) -- (1,1.75);

\draw { (1,1) -- (1,1.75) -- (1.75,1.75) -- (1.75,1)  -- (1,1) } [fill=gray!90]; 

\path (1.5,-1.3) node(caption){$\phi(\xi)=\xi_1^2+\xi_2^2$ };  

\draw[->] (3.6,2) -- (5.6,2) node[anchor=north]{};
\draw (6.1,2.3) node[anchor=east] {translation};


 \draw[->] (5.8,0) -- (10,0) node[anchor=north] {$\tilde{\xi}_1$};
\draw[->] (6,-0.2) -- (6,4)  node[anchor=east] {$\tilde{\xi}_2$};
 \draw (5.65,0) node[anchor=north] {O}
 (9.2,0) node[anchor=north] {1}
 (7.1,0) node[anchor=north] {\small $K^{-\frac12}$};

 \draw (6,3.2) node[anchor=east] {1}
      (6,0.75) node[anchor=east] {${K^{-\frac12}}$};

 \draw[thick] (6,0.75) -- (6.75,0.75)
             (6.75,0) -- (6.75,0.75)
              (9.2,0) -- (9.2,3.2)
             (6,3.2) -- (9.2,3.2);

 \draw { (6,0) -- (6.75,0) -- (6.75,0.75) -- (6,0.75)  -- (6,0) } [fill=gray!90]; 

\path (7.5,-1.3) node(caption){$\phi(\tilde{\xi})=\tilde{\xi}_1^2+\tilde{\xi}_2^2$ };  

\draw[->] (9.6,2) -- (11.6,2) node[anchor=north]{};
\draw (11.5,2.3) node[anchor=east] {scaling};


 \draw[->] (11.8,0) -- (16,0) node[anchor=north] {$\eta_1$};
\draw[->] (12,-0.2) -- (12,4)  node[anchor=east] {$\eta_2$};

  \draw (11.65,0) node[anchor=north] {O}
 (15.2,0) node[anchor=north] {1};

 \draw (12,3.2) node[anchor=east] {1};

  \draw[thick] (12,3.2) -- (15.2,3.2)
             (15.2,0) -- (15.2,3.2);

 \draw { (12,0) -- (15.2,0) -- (15.2,3.2) -- (12,3.2)  -- (12,0) } [fill=gray!90]; 

\path (13.5,-1.3) node(caption){$\phi(\eta)=\eta_1^2+\eta_2^2$ };  

\path (7.5,-2.5) node(caption){Fig. 2};
\end{tikzpicture}

 \end{center}

We notice that the standard parabolic rescaling is  inefficient for our purpose.
 We introduce the generalized rescaling which  matches  with  the surface $F^2_4$.
 $K\gg1$ is  independent of $R$ such that  $1\ll K \ll R^\varepsilon$.   For each  subset $\tau\subset [0,1]^{2}$ and phase function $\psi(\xi)=\xi^4_1+\xi^4_2$, we denote the Fourier extension operator by
\begin{equation}\label{eq-add-20}
  {\mathcal E}_{\tau}g(x):= \int_{\tau} g(\xi_1,\xi_2)e[x_1\xi_1+x_2\xi_2+x_3\psi(\xi_1,\xi_2)]\;d\xi_1\;d\xi_2.
\end{equation}
In view of the degenerate feature of the surface $F^2_4$, it suffices to consider the following three different cases:
\begin{enumerate}
  \item[(a)] $\tau=[0,K^{-\frac14}]\times[0,K^{-\frac14}]$;
  \vskip0.15cm
  \item[(b)] $\tau =[\lambda,\lambda+\lambda^{-1}K^{-\frac12}]\times[0,K^{-\frac14}]$;
  \vskip0.15cm
  \item[(c)]  $\tau = [\lambda,\lambda+\lambda^{-1}K^{-\frac12}]\times[\sigma,\sigma+\sigma^{-1}K^{-\frac12}]$.
\end{enumerate}
In case(a), the two principle curvatures of the surface $F^2_4$ are both degenerate. In case(b), only one of the principle curvatures is degenerate. In the last case, the surface $F^2_4$ has two positive principle curvatures.
\vskip0.15cm

{\bf (a). $\tau=[0,K^{-\frac14}]\times[0,K^{-\frac14}]$}. Under the change of variables
\begin{equation*}
  \left\{\begin{aligned}
    &\xi_1=K^{-\frac14}\eta_1,\\
    &\xi_2=K^{-\frac14}\eta_2
  \end{aligned}\right.
\end{equation*}
 we can rewrite \eqref{eq-add-20} as
\begin{equation*}
  \vert \mathcal E_{\tau}g(x)\vert=\Big| \int_{[0,1]^2}\tilde{g}(\eta_1,\eta_2)e[\tilde{x}_1\eta_1+\tilde{x}_2\eta_2+\tilde{x}_3\psi(\eta_1,\eta_2)]
  \;d\eta_1\;d\eta_2\Big|\\
  =\vert \mathcal E_{[0,1]^2}\tilde{g}(\tilde{x})\vert.,
\end{equation*}
where
\begin{equation*}
  \left\{\begin{aligned}
&\tilde{x}=\big(K^{-\frac14}x_1,K^{-\frac14}x_2,K^{-1}x_3\big),\\
&\tilde{g}(\eta_1,\eta_2)=K^{-\frac12}g\big(K^{-\frac14}\eta_1,K^{-\frac14}\eta_2\big).
\end{aligned}\right.\end{equation*}
See Fig.3 for the  scaling transformation.

\vskip0.3cm
\begin{center}
 \begin{tikzpicture}[scale=0.7]

\draw[->] (-0.2,0) -- (4,0) node[anchor=north] {$\xi_1$};
\draw[->] (0,-0.2) -- (0,4)  node[anchor=east] {$\xi_2$};

\draw (-0.15,0) node[anchor=north] {O}
(3.2,0) node[anchor=north] {$1$}
(1.2,0) node[anchor=north] {\small ${K^{-\frac14}}$};

\draw (0,3.2) node[anchor=east] {1}
       (0,1) node[anchor=east] {\small ${K^{-\frac14}}$};

\draw[thick] (0,3.2) -- (3.2,3.2)
             (3.2,0) -- (3.2,3.2)
             (1,0) -- (1,1)
             (0,1) -- (1,1);

\draw { (0,0) -- (1,0) -- (1,1) -- (0,1)  -- (0,0) } [fill=gray!90]; 

\path (1.5,-1.3) node(caption){$\psi(\xi)=\xi_1^4+\xi_2^4$ };  

\draw[->] (4.6,2) -- (6.6,2) node[anchor=north]{};
\draw (6.5,2.3) node[anchor=east] {scaling};


 \draw[->] (7.8,0) -- (12,0) node[anchor=north] {$\eta_1$};
\draw[->] (8,-0.2) -- (8,4)  node[anchor=east] {$\eta_2$};

  \draw (7.65,0) node[anchor=north] {O}
 (11.2,0) node[anchor=north] {1};

 \draw (8,3.2) node[anchor=east] {1};

  \draw[thick] (8,3.2) -- (11.2,3.2)
             (11.2,0) -- (11.2,3.2);

 \draw { (8,0) -- (11.2,0) -- (11.2,3.2) -- (8,3.2)  -- (8,0) } [fill=gray!90]; 

\path (9.5,-1.3) node(caption){$\psi(\eta)=\eta_1^4+\eta_2^4$ };  

\path (6.3,-2.5) node(caption){Fig. 3};
\end{tikzpicture}

 \end{center}

{\bf (b). $\tau =[\lambda,\lambda+\lambda^{-1}K^{-\frac12}]\times[0,K^{-\frac14}]$}, where $\lambda \in [K^{-\frac14},\frac{1}{2}]$ is dyadic.
we take change of variables
\begin{equation*}
  \left\{\begin{aligned}
    &\xi_1=\lambda+\frac{\eta_1}{\lambda K^{\frac12}},\\
    &\xi_2=\frac{\eta_2}{K^{\frac14}},
  \end{aligned}\right.
\end{equation*}
in \eqref{eq-add-20} and obtain
\begin{equation*}
  \vert \mathcal E_{\tau}g(x)\vert=\Big| \int_{[0,1]^2}\tilde{g}(\eta_1,\eta_2)e[\tilde{x}_1\eta_1+\tilde{x}_2\eta_2+\tilde{x}_3\psi_1(\eta_1,\eta_2)]
  \;d\eta_1\;d\eta_2\Big|\\
  =\vert \tilde{\mathcal{E}}_{[0,1]^2}\tilde{g}(\tilde{x})\vert,
\end{equation*}
where
\begin{equation*}\left\{\begin{aligned}
&\tilde{x}=\big(\lambda^{-1}K^{-\frac12}x_1+4\lambda^2K^{-\frac12}x_3,K^{-\frac14}x_2,K^{-1}x_3\big),\\
&\tilde{g}(\eta_1,\eta_2)=\lambda^{-1}K^{-3/4}g\big(\lambda+\frac{\eta_1}{\lambda K^{\frac12}},\frac{\eta_2}{K^{\frac14}}\big),\\
&\psi_1(\eta_1,\eta_2)
=(6\eta_1^2+4\lambda^{-2}K^{-\frac12}\eta_1^3+\lambda^{-4}K^{-1}\eta_1^4)
+\eta_2^4.
\end{aligned}\right.\end{equation*}
Here $\tilde{\mathcal E}_{[0,1]^2}$ denotes Fourier extension operators associated with phase functions of the form $\psi_1(\eta_1,\eta_2)=\phi_1(\eta_1)+\eta_2^4$ satisfying $\partial^2_1\phi_1 \sim 1, \vert \partial^3_1\phi_1 \vert \lesssim 1, \vert \partial^4_1\phi_1 \vert \lesssim 1$ on $[0,1]$.
 The  process of generalized rescaling  can be found in  Fig.4:

\vskip0.3cm

\begin{center}
 \begin{tikzpicture}[scale=0.7]
\draw[->] (-0.2,0) -- (4,0) node[anchor=north] {$\xi_1$};
\draw[->] (0,-0.2) -- (0,4)  node[anchor=east] {$\xi_2$};

\draw (-0.15,0) node[anchor=north] {O}
(3.2,0) node[anchor=north] {$1$}
(1,0) node[anchor=north] {$\lambda$,}
(2.2,0) node[anchor=north] {\small $\lambda+\tfrac{1}{\lambda K^{\frac12}}$};

\draw (0,3.2) node[anchor=east] {1}
       (0,1) node[anchor=east] {$K^{-\frac14}$};

\draw[thick] (0,3.2) -- (3.2,3.2)
             (3.2,0) -- (3.2,3.2)
             (1,0) -- (1,1)
             (1,1) -- (1.75,1)
             (1.75,0) -- (1.75,1);

\draw[dashed,thick] (0,1) -- (1,1); 

 \draw { (1,0) -- (1.75,0) -- (1.75,1) -- (1,1)  -- (1,0) } [fill=gray!90]; 

\path (1.5,-1.35) node(caption){$\psi(\xi)$ };  

\draw[->] (3.6,2) -- (5.6,2) node[anchor=north]{};
\draw (6.1,2.3) node[anchor=east] {translation};


 \draw[->] (5.8,0) -- (10,0) node[anchor=north] {$\tilde{\xi}_1$};
\draw[->] (6,-0.2) -- (6,4)  node[anchor=east] {$\tilde{\xi}_2$};
 \draw (5.65,0) node[anchor=north] {O}
 (9.2,0) node[anchor=north] {1}
 (7.1,0) node[anchor=north] {\small $\tfrac{1}{\lambda K^{\frac12}}$};

 \draw (6,3.2) node[anchor=east] {1}
      (6,1) node[anchor=east] {${K^{-\frac14}}$};

 \draw[thick] (6,1) -- (6.75,1)
             (6.75,0) -- (6.75,1)
              (9.2,0) -- (9.2,3.2)
             (6,3.2) -- (9.2,3.2);

 \draw { (6,0) -- (6.75,0) -- (6.75,1) -- (6,1)  -- (6,0) } [fill=gray!90]; 

\path (7.6,-1.35) node(caption){$\tilde{\psi}_1(\tilde{\xi})$};  

\draw[->] (9.6,2) -- (11.6,2) node[anchor=north]{};
\draw (11.5,2.3) node[anchor=east] {scaling};


 \draw[->] (11.8,0) -- (16,0) node[anchor=north] {$\eta_1$};
\draw[->] (12,-0.2) -- (12,4)  node[anchor=east] {$\eta_2$};

  \draw (11.65,0) node[anchor=north] {O}
 (15.2,0) node[anchor=north] {1};

 \draw (12,3.2) node[anchor=east] {1};

  \draw[thick] (12,3.2) -- (15.2,3.2)
             (15.2,0) -- (15.2,3.2);

 \draw { (12,0) -- (15.2,0) -- (15.2,3.2) -- (12,3.2)  -- (12,0) } [fill=gray!90]; 

\path (14.0,-1.35) node(caption){$\psi_1(\eta)$ };  

\path (7.5,-2.5) node(caption){Fig. 4};
\end{tikzpicture}
 \end{center}

where
\begin{equation*}\left\{\begin{aligned}
&\psi(\xi)=\xi_1^4+\xi_2^4,\\
&\tilde{\psi}_1(\tilde{\xi})=(6\tilde{\xi}_1^2+4\lambda^{-2}\tilde{\xi}_2^3
+\lambda^{-4}\tilde{\xi}_1^4)+\tilde{\xi}_2^4\\
&{\psi}_1(\eta)=(6\eta_1^2+4\lambda^{-2}K^{-\frac12}\eta_2^3+\lambda^{-4}K^{-1}\eta_2^4)+\eta_2^4,\\
\end{aligned}\right.\end{equation*}
\vskip 0.1in

{\bf (c)}. $\tau = [\lambda,\lambda+\lambda^{-1}K^{-\frac12}]\times[\sigma,\sigma+\sigma^{-1}K^{-\frac12}]$, where $ \lambda,\sigma \in [K^{-\frac14}, \frac{1}{2}]$ are both dyadic numbers. From the discussion in Subsection \ref{subsec:difgeo}, we know that the Gaussian curvature of $\gamma(u,v)$ is roughly constant in the region $$\Omega_{\lambda,\sigma}=\Big\{(u,v):\; \lambda \leq u \leq 2\lambda,\;\;\sigma < v < 2\sigma\Big\}.$$
Therefore, it is reasonable to divide $\Omega_{\lambda,\sigma}$ into $\lambda^2 K^{\frac12}\times\sigma^2 K^{\frac12}$ pieces equally
$$\Omega_{\lambda,\sigma}=\bigcup_{1\le j\le\lambda^2K^{\frac12}\atop 1\le m\le \sigma^2K^{\frac12}}
\big[\lambda+(j-1)\lambda^{-1}K^{-\frac12}, \lambda+j\lambda^{-1}K^{-\frac12} \big]\times\big[\sigma+(m-1)\sigma^{-1}K^{-\frac12}, \sigma+m\sigma^{-1}K^{-\frac12}\big].$$
with dimension $\lambda^{-1}K^{-\frac12}\times\sigma^{-1}K^{-\frac12}$.
 Without loss of generality, we can implement generalized rescaling to each piece in the union such as $\tau$.
   We take the change of variables
\begin{equation*}\left\{
  \begin{aligned}
    &\xi_1=\lambda+\frac{\eta_1}{\lambda K^{\frac12}}\\
    &\xi_2=\sigma+\frac{\eta_2}{\sigma K^{\frac12}},
  \end{aligned}\right.
\end{equation*}
to get
\begin{equation*}
  \vert \mathcal E_{\tau}g(x)\vert=\Big| \int_{[0,1]^2}\tilde{g}(\eta_1,\eta_2)e[\tilde{x}_1\eta_1+\tilde{x}_2\eta_2+\tilde{x}_3\psi_0(\eta_1,\eta_2)]
  \;d\eta_1\;d\eta_2\Big|=:\vert {\mathcal E}^{\rm Parp}_{[0,1]^2}\tilde{g}(\tilde{x})\vert.
\end{equation*}
where
\begin{equation*}\left\{
  \begin{aligned}
&\tilde{x}=\big(\lambda^{-1}K^{-\frac12}x_1+4\lambda^2K^{-\frac12}x_3,\sigma^{-1}K^{-\frac12}x_2+4\sigma^2K^{-\frac12}x_3,K^{-1}x_3\big),\\
&\tilde{g}(\eta_1,\eta_2)=\lambda^{-1}\sigma^{-1}K^{-1}g\Big(\lambda+\frac{\eta_1}{\lambda K^{\frac12}},\sigma+\frac{\eta_2}{\sigma K^{\frac12}}\Big),\\
&\psi_0(\eta_1,\eta_2)=(6\eta_1^2+4\lambda^{-2}K^{-\frac12}\eta_1^3+\lambda^{-4}K^{-1}\eta_1^4)
+(6\eta_2^2+4\sigma^{-2}K^{-\frac12}\eta_2^3+\sigma^{-4}K^{-1}\eta_2^4).
\end{aligned}\right.
\end{equation*}
Here $\mathcal{E}^{\rm Parp}_{[0,1]^2}$ denotes Fourier extension operators associated with the class of surfaces with positive definite second fundamental form. Note that $\lambda,\sigma \geq K^{-\frac14}$ and $0\leq \xi_i \leq 1(i=1,2)$,  we conclude that the phase functions are of the form $\psi_0(\eta_1,\eta_2)= \phi_1(\eta_1)+\phi_2(\eta_2)$ satisfying
$$\partial^2_i\phi_i(\eta_i)\sim 1, \;\vert \partial^3_i\phi_i(\eta_i)\vert \lesssim 1,\; \vert \partial^4_i\phi_i(\eta_i)\vert \lesssim 1, \;\;\;\text{on} \;[0,1],\; i=1,2.$$

\vskip0.25cm
\begin{center}
 \begin{tikzpicture}[scale=0.7]
\draw[->] (-0.2,0) -- (4,0) node[anchor=north] {$\xi_1$};
\draw[->] (0,-0.2) -- (0,4)  node[anchor=east] {$\xi_2$};

\draw (-0.15,0) node[anchor=north] {O}
(3.2,0) node[anchor=north] {$1$}
(1,0) node[anchor=north] {$\lambda$,}
(2.1,0) node[anchor=north] {\small $\lambda+\tfrac{1}{\lambda K^{\frac12}}$};

\draw (0,3.2) node[anchor=east] {1}
      (0,1) node[anchor=east] {$\sigma$}
       (0,1.75) node[anchor=east] {\small $\sigma+\frac{1}{\sigma K^{\frac12}}$};

\draw[thick] (0,3.2) -- (3.2,3.2)
             (3.2,0) -- (3.2,3.2)
             (1,1) -- (1,1.75)
             (1,1) -- (1.75,1)
             (1,1.75) -- (1.75,1.75)
             (1.75,1) -- (1.75,1.75);

\draw[dashed,thick] (1,0) -- (1,1) 
                     (1.75,0) -- (1.75,1)
                     (0,1) -- (1,1)
                     (0,1.75) -- (1,1.75);

\draw { (1,1) -- (1,1.75) -- (1.75,1.75) -- (1.75,1)  -- (1,1) } [fill=gray!90]; 

\path (1.5,-1.35) node(caption){$\psi(\xi)$ };  

\draw[->] (3.6,2) -- (5.6,2) node[anchor=north]{};
\draw (6.1,2.3) node[anchor=east] {translation};


 \draw[->] (5.8,0) -- (10,0) node[anchor=north] {$\tilde{\xi}_1$};
\draw[->] (6,-0.2) -- (6,4)  node[anchor=east] {$\tilde{\xi}_2$};
 \draw (5.65,0) node[anchor=north] {O}
 (9.2,0) node[anchor=north] {1}
 (7.1,0) node[anchor=north] {\small $\tfrac{1}{\lambda K^{\frac12}}$};

 \draw (6,3.2) node[anchor=east] {1}
      (6,0.75) node[anchor=east] {\small $\tfrac{1}{\sigma K^{\frac12}}$};

 \draw[thick] (6,0.75) -- (6.75,0.75)
             (6.75,0) -- (6.75,0.75)
              (9.2,0) -- (9.2,3.2)
             (6,3.2) -- (9.2,3.2);

 \draw { (6,0) -- (6.75,0) -- (6.75,0.75) -- (6,0.75)  -- (6,0) } [fill=gray!90]; 

\path (7.5,-1.35) node(caption){$\tilde{\psi}_0(\tilde{\xi})$ };  

\draw[->] (9.6,2) -- (11.6,2) node[anchor=north]{};
\draw (11.5,2.3) node[anchor=east] {scaling};


 \draw[->] (11.8,0) -- (16,0) node[anchor=north] {$\eta_1$};
\draw[->] (12,-0.2) -- (12,4)  node[anchor=east] {$\eta_2$};

  \draw (11.65,0) node[anchor=north] {O}
 (15.2,0) node[anchor=north] {1};

 \draw (12,3.2) node[anchor=east] {1};

  \draw[thick] (12,3.2) -- (15.2,3.2)
             (15.2,0) -- (15.2,3.2);

 \draw { (12,0) -- (15.2,0) -- (15.2,3.2) -- (12,3.2)  -- (12,0) } [fill=gray!90]; 

\path (13.5,-1.35) node(caption){${\psi}_0(\eta)$ };  

\path (7.5,-2.5) node(caption){Fig. 5};
\end{tikzpicture}
 \end{center}
\begin{equation*}\left\{
  \begin{aligned}
&\psi(\xi)=\xi_1^4+\xi_2^4,\\
&\tilde{\psi}_0(\tilde{\xi})=(6\tilde{\xi}_1^2+4\lambda^{-2}\tilde{\xi}_1^3+\lambda^{-4}\tilde{\xi}_1^4)+(6\tilde{\xi}_2^2
+4\lambda^{-2}\tilde{\xi}_2^3+\lambda^{-4}\tilde{\xi}_2^4)\\
&\psi_0(\eta)=(6\eta_1^2+4\lambda^{-2}K^{-\frac12}\eta_1^3+\lambda^{-4}K^{-1}\eta_1^4)+
(6\eta_2^2+4\lambda^{-2}K^{-\frac12}\eta_2^3+\lambda^{-4}K^{-1}\eta_2^4).
\end{aligned}\right.
\end{equation*}

\vskip 0.1in

From the above discussion, we can construct decomposition of  $[0,1]^{2}$ associated with the phase
function $\psi(\xi)=\xi_1^4+\xi_2^4$.
At first, we divide $[0,1]$ into
\begin{equation*}
  [0,1]=\bigcup I_j ,
\end{equation*}
where $I_0=[0,K^{-\frac14}]$ and
$$I_j=[2^{j-1}K^{-\frac14},2^{j}K^{-\frac14}], \;\;\text{for}\;\; 1\leq j \leq \big[\tfrac{1}{4}\log_{2}K\big].$$

Secondly, for $j\geq 1$, we divide  each $I_j$ further into
\[I_j=\bigcup_{\mu=1}^{2^{2(j-1)}} I_{j,\mu},\]
where
$$I_{j,\mu}=[2^{j-1}K^{-\frac14}+(\mu-1)2^{-(j-1)}K^{-\frac14},2^{j-1}K^{-\frac14}+ \mu 2^{-(j-1)}K^{-\frac14}].$$
Thus we have the following  decomposition
\[ [0,1]^{2}=\bigcup \tau, \;\;\; \tau = I_{j_1,\mu_1}\times I_{j_2,\mu_2}.\]
We call it  a $K-$regular decomposition, denoted by $\mathcal{F}_3(K,4)$,  adapted to finite type surface of order $4$
in $\mathbb{R}^3$.
 Analogous decomposition appeared in the  restriction estimates for certain conical surfaces of finite type in $\mathbb{R}^3$ by Buschenhenke in \cite{Bus}.

\begin{remark}
From the above discussion, we see that the rescaling technique fails if one decomposes the region $[0,1]^2$ into pieces equally. Thus, it is necessary to employ the decomposition in this paper
 to tackle with harmonic analysis problems associated with surfaces of finite type.
\end{remark}

\begin{remark}\label{equ:rem:scale}
Since the estimate we considered are localized on $B_R$ in the physical space, by the uncertainty principle, the local constant property holds at scale $R^{-1}$ on the Fourier side. For each $\lambda^{-1}R^{-\frac12}\times\sigma^{-1}R^{-\frac12}$-region $\theta$,  we denote $\bar\theta$ by
$$\bar\theta=\Big\{(\xi_1,\xi_2,\xi_3):\;\; (\xi_1,\xi_2)\in \theta, \;\big|\xi_3-(\xi_1^4+\xi_2^4)\big|< R^{-1}\Big\}$$
For each $\theta\in \mathcal F_3(R,4)$, we can regard the associated $\bar{\theta}$ as a rectangular box which is called a slab. For each slab $\bar{\theta}$,  $\mathbb T(\bar\theta)$ denotes a finitely-overlapping
collection of $\sim \lambda R^{\frac12}\times \sigma R^{\frac12}\times R$ rectangles which cover $\R^3$ and are orientated in the direction of $\bar\theta$. The decomposition $\mathcal F_3(R,4)$ ensures that the tubes $T\in \mathbb{T}(\bar{\theta})$ have a common length $R$ in the long side direction,  which is indispensable for establishing wave packet decomposition in Subsection \ref{subs:wavepack} and the Kakeya-type estimate in Subsection \ref{subs:square} below.
\end{remark}




\section{Proof of Theorem \ref{thm:main}}

From now on,  we  denote
\begin{equation*}
  \Sigma:=\big\{(\xi_1,\xi_2,\psi(\xi)): (\xi_1,\xi_2)\in [0,1]^2\big\},\;\;~\psi(\xi)=\xi_1^4+\xi_2^4
\end{equation*}
in $\mathbb{R}^3$. The corresponding Fourier extension operator becomes
\[\mathcal Eg(x):=\int_{[0,1]^2}g(\xi)e(x_1\xi_1+x_2\xi_2+x_3\psi(\xi))\;d\xi_1\;d\xi_2,\]
where $e(t):=e^{2\pi i t}$ for $t\in \mathbb{R}^1$.
\vskip0.12cm

 Unwinding the definition of $\Sigma$ and $\mathcal Eg$, Theorem \ref{thm:main} can be rewritten as
\begin{theorem}
  For any given $\varepsilon >0$ and any radius $R\geq 1$, the inequality
\begin{equation}\label{equ:redgoal}
  \|\mathcal Eg\|_{L^p(B_R)}\leq C_{\varepsilon}R^{\varepsilon}\|g\|_{L^{\infty}(\Sigma)}
\end{equation}
holds for all $p>3.3$.
\end{theorem}

Let $Q_p(R)$ denote the least number such that
\begin{equation}\label{equ:defqpr}
  \|\mathcal Eg\|_{L^p(B_R)}\leq Q_p(R)\|g\|_{L^{\infty}(\Sigma)}.
\end{equation}

Let $1\ll K\ll R^{\varepsilon}$ for any fixed $\varepsilon>0$,
we divide $[0,1]^2$ into $\bigcup\limits_{j=0}^3\Omega_j$, as in Fig.6, where
\begin{align*}
&\Omega_0:=[K^{-\frac14},1]\times[K^{-\frac14},1],\\
&\Omega_1:=[K^{-\frac14},1]\times[0,K^{-\frac14}],\\
&\Omega_2:=[0,K^{-\frac14}]\times[K^{-\frac14},1],\\
&\Omega_3:=[0,K^{-\frac14}]\times[0,K^{-\frac14}],\\
&g_{\Omega_j}=g\big|_{\Omega_j},\;0\leq j\leq 3.
\end{align*}
\begin{center}
 \begin{tikzpicture}[scale=0.6]

\draw[->] (-0.2,0) -- (5,0) node[anchor=north]   {$\xi_1$};
\draw[->] (0,-0.2) -- (0,5)  node[anchor=east] {$\xi_2$};

 \draw[thick] (0,4) -- (4,4)
 (4,0)--(4,4)
 (0,1)--(4,1)
 (1,0)--(1,4) ;

 \draw (-0.15,0) node[anchor=north] {O}
      (1.35,-0.2) node[anchor=north] {$K^{-\frac14}$}
     (4,-0.08) node[anchor=north] {$1$}
    (2.5,2.5) node[anchor=north] {$\Omega_0$}
     (2.5,0.8) node[anchor=north] {$\Omega_1$}
     (0.5,2.5) node[anchor=north] {$\Omega_2$}
      (0.5,0.8) node[anchor=north] {$\Omega_3$};

 \draw  (-0.35, 1) node[anchor=east] {$K^{-\frac14}$}
       (0,4) node[anchor=east] {$1$};


\path (2.5,-2.0) node(caption){Fig. 6};
 \end{tikzpicture}
\end{center}
In this setting, we have
\begin{equation}\label{equ:egr03}
   \|\mathcal Eg\|_{L^p(B_R)}\leq \sum_{j=0}^{3}\big\|\mathcal Eg_{\Omega_j}\big\|_{L^p(B_R)}.
\end{equation}
Since the surface corresponding to the region $\Omega_0$ possesses two positive principle curvatures with lower bounds depending only on $K$,
 we have by \cite[Theorem 0.1]{Guth}
\begin{equation}\label{equ:ome0est}
  \|\mathcal Eg_{\Omega_0}\|_{L^p(B_R)}\leq C(K)C(\varepsilon)R^{\varepsilon}\|g\|_{L^{\infty}(\Sigma)},
\end{equation}
for $p>3.25$.

For $\Omega_3$, using rescaling as in Subsection 2.2 (a) and definition \eqref{equ:defqpr}, we
easily conclude that
\begin{equation}\label{equ:omeest}
  \|\mathcal Eg_{\Omega_3}\|_{L^p(B_R)}\leq CK^{\frac{3}{2p}-\frac{1}{2}}Q_p\Big(\tfrac{R}{K^{\frac14}}\Big)\|g\|_{L^{\infty}(\Sigma)}.
\end{equation}

For $\Omega_1$ and $\Omega_2$, it suffices to consider the estimate for $\Omega_1$-part by
 symmetry. We decompose $\Omega_1$ further into
\[\Omega_1=\bigcup \Omega_{\lambda}, \quad \Omega_{\lambda}=[\lambda,2\lambda]\times [0,K^{-\frac14}],\]
for dyadic $\lambda$ satisfying $K^{-\frac14}\leq \lambda \leq \frac{1}{2}$.
Let $\Omega_{\lambda,0}=[\lambda,2\lambda]\times [0,R^{-\frac14}]$, and  write
\begin{equation*}
\Omega_{\lambda}=\Omega_{\lambda,0}\bigcup \big(\cup_{\sigma}\Omega_{\lambda,\sigma}\big),
\end{equation*}
 where
    $\Omega_{\lambda,\sigma}:=[\lambda,2\lambda]\times([0, \sigma]\setminus [0, K^{-\frac14}\sigma])$ for $\sigma=K^{\frac{\ell}{4}}R^{-\frac14}$ with $1\leq \ell \leq [\log_K R]-1$. Then, we have
\begin{align}\label{equ:ome1part}
\|\mathcal Eg_{\Omega_1}\|_{L^p(B_R)}\leq\sum_{\lambda}\Big[\|\mathcal Eg_{\Omega_{\lambda,0}}\|_{L^p(B_R)}
+\sum_{\sigma}\|\mathcal Eg_{\Omega_{\lambda,\sigma}}\|_{L^p(B_R)}\Big].
\end{align}
We claim that
\begin{align}\label{equ:omelamsig}
   \|\mathcal Eg_{\Omega_{\lambda,\sigma}}\|_{L^p(B_R)}
   \lesssim_{\varepsilon}C(K)R^{\varepsilon}\|g\|_{L^{\infty}(\Sigma)},\; \;\;\;\;p>3.3\;
\end{align}
and
\begin{equation}\label{equ:claimlambda0}
  \|\mathcal{E}g_{\Omega_{\lambda,0}}\|_{L^p(B_R)}\lesssim C(K)R^{\frac{9}{4p}-\frac{3}{4}+\varepsilon}\|g\|_{L^{\infty}(\Sigma)},\;\;\; p\in[2,4],
\end{equation}
where  the implicit constant is independent of $\lambda$ and $\sigma$. Plugging the claim \eqref{equ:omelamsig} and \eqref{equ:claimlambda0} into \eqref{equ:ome1part}, we obtain
$$\|\mathcal Eg_{\Omega_1}\|_{L^p(B_R)}\leq C(K)C_{\varepsilon}R^{\varepsilon}\|g\|_{L^{\infty}(\Sigma)},\;\;\;p>3.3.$$
This inequality together with \eqref{equ:egr03}-\eqref{equ:omeest} implies
for $p>3.3$
$$\|\mathcal Eg\|_{L^p(B_R)}\leq \Big[C(K)C(\varepsilon)R^{\varepsilon}+2C(K)C_{\varepsilon}R^{\varepsilon}
+CK^{\frac{3}{2p}-\frac{1}{2}}Q_p\Big(\tfrac{R}{K^{\frac14}}\Big)\Big]\|g\|_{L^{\infty}(F_4^2)}, $$
which implies
$$Q_p(R)\leq C(K)C(\varepsilon)R^{\varepsilon}+2C(K)C_{\varepsilon}R^{\varepsilon}
+CK^{\frac{3}{2p}-\frac{1}{2}}Q_p\Big(\tfrac{R}{K^{\frac14}}\Big).$$
Hence, this recurrence inequality yields Theorem \ref{thm:main}.

Now, it remains to verify the claim \eqref{equ:omelamsig}  and \eqref{equ:claimlambda0}. For this purpose, we first estimate the contribution from $\Omega_{\lambda,\sigma}$ with $\sigma=K^{\frac{\ell}{4}}R^{-\frac14}$.

\subsection{Restriction estimates for $\Omega_{\lambda,\sigma}$}

In this subsection, we will prove the  claim \eqref{equ:omelamsig}
 for $K^{-\frac14}\leq \lambda \leq \frac{1}{2}$ and $R^{-\frac14}\leq \sigma \leq K^{-\frac14}$.
\vskip0.15cm

Firstly, we establish a decoupling inequality as follows

\begin{proposition}\label{prop:decoupling}
Let $\Sigma_{\lambda,\sigma}:=\{(\xi_1,\xi_2,\xi_3): (\xi_1,\xi_2)\in [\lambda,2\lambda]\times[0,\sigma], \xi_3=\xi^4_1+\xi^4_2\}$.
 For each suitable function $F$ satisfying ${\rm supp}\;\hat{F}\subset \mathcal{N}_{\sigma^4}(\Sigma_{\lambda,\sigma})$ with $\mathcal{N}_{{R^{-1}}}(\Sigma_{\lambda,\sigma})$ being $R^{-1}$-neighborhood of $\Sigma_{\lambda,\sigma}$, there holds
\[\|F\|_{L^{p}(B_{\sigma^{-4}})}\lesssim_{\varepsilon}\sigma^{-\varepsilon}\Big(\sum_{\tau}\|F_{\tau}\|^2_{L^p(B_{\sigma^{-4}})}\Big)^{\frac12},
\;\;\; 2\leq p \leq 6,\]
where  $\tau$ denotes $\lambda^{-1}\sigma^2 \times \sigma \times \sigma^4$-slab contained in $\mathcal{N}_{\sigma^4}(\Sigma_{\lambda,\sigma})$.
\end{proposition}

To prove Proposition \ref{prop:decoupling}, it is sufficient to show the following lemma by freezing the $x_2$ variable. Without of loss generality, we assume that the ball $B_{\sigma^{-4}}$ is centred at the origin.
\begin{lemma}\label{lem:curve}
For $\Gamma_{\lambda}=\{(t,t^4): t\in [\lambda,2\lambda]\}$ and any suitable function $G$ satisfying ${\rm supp}\;\hat{G}\subset\mathcal{N}_{\sigma^4}(\Gamma_{\lambda})$, it holds
\begin{equation}\label{equ:curvedec}
  \|G\|_{L^p(B^2_{\sigma^{-4}})}\lesssim_{\varepsilon}\sigma^{-\varepsilon} \big(\sum_{\theta}\big\| G_{\theta}\big\|^2_{L^p(B^2_{\sigma^{-4}})}\big)^{1/2},\;\;\; 2\le p\le6,
\end{equation}
 where $\theta: \lambda^{-1}\sigma^2\times \sigma^4$-slab.
\end{lemma}

With Lemma \ref{lem:curve} in hand, we prove Proposition \ref{prop:decoupling} as follows. We denote $F(\cdot,x_2,\cdot)$ by $G$. By the arguments in \cite{Guth18}, it is easy to see that ${\rm supp}\;\hat{G}$ is contained in the projection of ${\rm supp}\;\hat{F}$ on the plane $\xi_2=0$, that is, ${\rm supp}\;\hat{G}\subset \mathcal{N}_{\sigma^{-4}}(\Gamma_{\lambda})$.  Integrating the both sides of \eqref{equ:curvedec} with respect to $x_2$-variable from $-\sigma^{-4}$ to $\sigma^{-4}$, we obtain Proposition 3.2.

\begin{proof}[{\bf Proof of Lemma \ref{lem:curve}:}]
It is equivalent to prove for any suitable function $h(t)$ on $[\lambda,2\lambda]$
\begin{equation}\label{equ:equivalentcurvedec}
\Vert E_{[\lambda,2\lambda]}h \Vert_{L^p(B^2_{\sigma^{-4}})}\lesssim_{\varepsilon}\sigma^{-\varepsilon}\Big(\sum_{I:\lambda^{-1}\sigma^2-interval}\Vert Eh_{I} \Vert^2_{L^p(B^2_{\sigma^{-4}})}\Big)^{1/2},
\end{equation}
where \[(E_{[\lambda,2\lambda]}h)(y):=\int_{[\lambda,2\lambda]}h(t)e(y_1t+y_2t^4)dt.\]
  Using variable substitution as $t=\lambda u+\lambda$ in \eqref{equ:equivalentcurvedec}, we have
  \[\vert(E_{[\lambda,2\lambda]}h)(y)\vert=\vert(E^{\rm Parp}_{[0,1]}\tilde{h})(\tilde{y})\vert,\]
  where
  \[\tilde{h}(u)=\lambda h(\lambda u+\lambda), \;\;\;\tilde{y}=(\lambda y_1+4\lambda^4y_2,\lambda^4y_2),\]
  and $E^{\rm Parp}_{[0,1]}$ denotes the Fourier extension operator associated with the non-degenerate curve: \[\gamma(u)=(u,6u^2+4u^3+u^4), u\in [0,1].\]
  Let $\mathcal{L}$ denote the mapping
   \[y \rightarrow (\lambda y_1+4\lambda^4y_2,\lambda^4y_2)=:(\tilde{y}_1, \tilde{y}_2)=\tilde{y}.\]
   Therefore, one has
  \[\Vert E_{[\lambda,2\lambda]}h \Vert^p_{L^p(B^2_{\sigma^{-4}})}=\lambda^{-5}\Vert E^{\rm Parp}_{[0,1]}\tilde{h} \Vert^p_{L^p(\mathcal{L}(B^2_{\sigma^{-4}}))},\]
  which equals to
  \[\lambda^{-5}\sum_{B^2_{\lambda^4 \sigma^{-4}}\subset \mathcal{L}(B^2_{\sigma^{-4}})}\Vert E^{\rm Parp}_{[0,1]}\tilde{h} \Vert^p_{L^p(B^2_{\lambda^4\sigma^{-4}})}.\]
  Using Bourgain-Demeter's decoupling inequality for curves with non-vanishing curvature in \cite{BD}, one has for $2\leq p \leq 6$
  \begin{align*}
  &\sum_{B^2_{\lambda^4\sigma^{-4}}\subset \mathcal{L}(B^2_{\sigma^{-4}})}\Vert E^{\rm Parp}_{[0,1]}\tilde{h} \Vert^p_{L^p(B^2_{\lambda^4\sigma^{-4}})} \\
  \lesssim_{\varepsilon}&\sigma^{-\varepsilon} \sum_{B^2_{\lambda^4\sigma^{-4}}\subset \mathcal{L}(B^2_{\sigma^{-4}})} \Big(\sum_{\tilde{I}:\lambda^{-2}\sigma^2-interval} \big\| E^{\rm Parp}_{\tilde{I}}\tilde{h}\big\|^2_{L^p(B^2_{{\lambda^4}\sigma^{-4}})}\Big)^{p/2}\\
    \le&C_{\varepsilon}\sigma^{-\varepsilon}\Big(\sum_{\tilde{I}:\lambda^{-2}\sigma^2-interval}\big\| E^{\rm Parp}_{\tilde{I}}\tilde{h}\big\|^2_{L^p(\mathcal{L}(B^2_{\sigma^{-4}}))}\Big)^{p/2}.\end{align*}
  Taking the inverse of $\mathcal{L}$, we get the desired estimate \eqref{equ:equivalentcurvedec}.
\end{proof}

\begin{lemma}\label{lem:omesing}
For all $p>3.3$, it holds
\begin{equation}\label{equ:lamsigma}
  \|{\mathcal E}g_{\Omega_{\lambda,\sigma}}\|_{L^p(B_R)}
  \lesssim_{\varepsilon}C(K)R^{\varepsilon}\|g\|_{L^{\infty}(\Sigma)},
\end{equation}
where $C(K)$ is a fixed power of $K$.
\end{lemma}

\begin{proof} Due to Theorem 1.2 in  \cite{BMV}, we have
\[\|{\mathcal E}g_{\Omega_{\lambda,\sigma}}\|_{L^{p}(B_R)}\lesssim_{\varepsilon}
R^{\varepsilon}\|g_{\Omega_{\lambda,\sigma}}\|_{L^{\frac{p}{p-3}}}, \;\;\;\text{for all}\; \;p>\tfrac{10}{3}.\]
This estimate together with $\vert {\Omega}_{\lambda,\sigma} \vert \leq \lambda\sigma$
yields
\begin{equation}\label{equ:lamsigma-00}\|{\mathcal E}g_{\Omega_{\lambda,\sigma}}\|_{L^{\frac{10}{3}+}(B_R)}
\lesssim_{\varepsilon,\lambda}R^{\varepsilon}\sigma^{\frac{1}{10}}\|g\|_{L^{\infty}(\Sigma)}.\end{equation}
Hence, \eqref{equ:lamsigma} can be reduced to prove the following claim:
\begin{equation}\label{equ:claim1}
\|{\mathcal E}g_{\Omega_{\lambda,\sigma}}\|_{L^{3.25}(B_R)}
\lesssim_{\varepsilon,\lambda}C(K)R^{\varepsilon}\sigma^{-\frac{2}{13}}\|g\|_{L^{\infty}(\Sigma)}.
\end{equation}
In fact, by interpolation with \eqref{equ:lamsigma-00}, we deduce that
\[\|\mathcal Eg_{\Omega_{\lambda,\sigma}}\|_{L^p(B_R)}\lesssim_{\varepsilon}C(K)R^{\varepsilon}\|g\|_{L^{\infty}(\Sigma)}, \;\;\; p>3.3.\]
It remains to verify \eqref{equ:claim1}. To this end, we employ an equivalent form of decoupling inequality in Proposition \ref{prop:decoupling}
\begin{equation}\label{equ:decoupling}
\|{\mathcal E}g_{\Omega_{\lambda,\sigma}}\|_{L^{p}(B_{\sigma^{-4}})}\lesssim_{\varepsilon}\sigma^{-\varepsilon}\Big(\sum_{\tau}\|{\mathcal E}g_{\tau\cap \Omega_{\lambda,\sigma}}\|^2_{L^p(B_{\sigma^{-4}})}\Big)^{1/2},
\end{equation}
where {the summation is taken over $\lambda^{-1}\sigma^2 \times \sigma$-rectangles}. Summing over all the balls $B_{\sigma^{-4}}\subset B_R$, one has
\[\|{\mathcal E}g_{\Omega_{\lambda,\sigma}}\|_{L^{p}(B_R)}\lesssim_{\varepsilon}\sigma^{-\varepsilon}\Big(\sum_{\tau}\|{\mathcal E}g_{\tau\cap \Omega_{\lambda,\sigma}}\|^2_{L^p(B_R)}\Big)^{1/2}.\]
By the so-called generalized rescaling technique, it is not difficult to show
\begin{equation}\label{byrecaling}
\|{\mathcal E}g_{\tau \cap \Omega_{\lambda,\sigma}}\|_{L^p(B_R)}\lesssim_{\varepsilon}C(K)R^{\varepsilon}\sigma^{3-\frac{7}{p}}\|g\|_{L^{\infty}(\Sigma)},
\;\;\; p\geq 3.25,
\end{equation}
where $C(K)$ is a fixed power of $K$.
In fact, making use of changing variables
$$\xi_1=\lambda+\lambda^{-1}\sigma^2\eta_1, \;\;\;\xi_2=\sigma\eta_2,$$
 we have
\[\vert {\mathcal E}g_{\tau \cap \Omega_{\lambda,\sigma}}\vert = \Big\vert \int_{[0,1]\times [K^{-1/4},1]}\tilde{g}(\eta_1,\eta_2)e[\tilde{x}_1\eta_1+\tilde{x}_2\eta_2+\tilde{x}_3(\phi_1(\eta_1)+\eta^4_2)]d\eta_1d\eta_2\Big\vert=:\vert{\mathcal E}^{parp}\tilde{g}(\tilde{x})\vert,\]
where
\begin{equation*}\left\{\begin{aligned}
&\tilde{x}_1:=\lambda^{-1}\sigma^2x_1+4\lambda^2\sigma^2x_3, \;\;\; \tilde{x}_2:=\sigma x_2, \;\;\; \tilde{x}_3:=\sigma^4x_3,\\
&\tilde{g}(\eta_1,\eta_2):=\lambda^{-1}\sigma^3 g(\lambda+\lambda^{-1}\sigma^2\eta_1,\sigma\eta_2),\\
&\phi_1(\eta_1):=6\eta^2_1+4\lambda^{-2}\sigma^2\eta^3_1+\lambda^{-4}\sigma^4\eta^4_1.\end{aligned}\right.\end{equation*}
Clearly, we have $\partial^2_1 \phi_1\approx 1$. Since $K^{-1/4}\leq \lambda \leq \frac{1}{2}$ and $\sigma \leq K^{-1/4}$, it follows that
\[\vert \partial^3_1 \phi_1 \vert \lesssim 1,\;\; \; \vert \partial^4_1 \phi_1 \vert \lesssim 1.\]
This together with $K^{-1/4}\leq \eta_2 \leq 1$ implies that the  phase function $\psi_1(\eta):=\phi_1(\eta_1)+\eta^4_2$ is admissible so that we can apply Guth's restriction estimate for perturbed paraboloid in 3D to $\vert{\mathcal E}^{parp}\tilde{g}(\tilde{x})\vert$. Therefore, one has
\[\Vert{\mathcal E}^{parp}\tilde{g}(\tilde{x})\Vert_{L^p(B_{\sigma R})}\lesssim_{\varepsilon}C(K)R^{\varepsilon}\Vert \tilde{g} \Vert_{L^{\infty}}, \;\;\;\; p\geq 3.25,\]
where $C(K)$ is a fixed power of $K$.
Note that
\begin{align*}\big\|{\mathcal E}g_{\tau \cap \Omega_{\lambda,\sigma}}\big\|_{L^p(B_R)}\lesssim_{\lambda}\sigma^{-\frac{7}{p}}\Vert{\mathcal E}^{parp}\tilde{g}(\tilde{x})\Vert_{L^p(B_{\sigma R})}\lesssim_{\lambda}C(K)R^{\varepsilon}\sigma^{3-\frac7p}\Vert g \Vert_{L^{\infty}},\end{align*}
we deduce \eqref{byrecaling}, as required. Plugging \eqref{byrecaling} into \eqref{equ:decoupling},  we derive
by making use of  ${\rm Card} \{\tau \} \approx \lambda^2 \sigma^{-2}$
\[\|{\mathcal E}g_{\Omega_{\lambda,\sigma}}\|_{L^{3.25}(B_R)}
\lesssim_{\varepsilon,\lambda}C(K)R^{\varepsilon}\sigma^{-\frac{2}{13}}\|g\|_{L^{\infty}(\Sigma)},\;\;\; K^{-1/4}\leq \lambda \leq \frac{1}{2}.\]
 This completes the proof of Lemma \ref{lem:omesing}.
\end{proof}

\vskip 0.12in

Now,  we show the claim \eqref{equ:claimlambda0} from the contribution of $\Omega_{\lambda,0}$-part by  utilizing  the wave packet decomposition and a square function estimate.

\subsection{Wave packet decomposition associated with $\Omega_{\lambda,0}$}\label{subs:wavepack}

In this subsection, we adopt notations from Hickman-Vitturi's lecture  \cite{HV}.
Let $\phi(x)$  be a bump function whose Fourier transform is supported on $[-\frac{1}{2},\frac{1}{2}]^3$
satisfies  $\widehat{\phi}=1$ on $[-\frac{1}{4},\frac{1}{4}]^3$.
For any tube $T$, denote by  $a_T$ an affine transformation whose linear part has determinant $\sim\vert T \vert$ which maps $[-\frac{1}{4},\frac{1}{4}]^3$ to $T$ bijectively and define $\phi_{T}:= \phi \circ a_{T}^{-1}$, $T\in \mathbb{T}(\bar{\theta})$ for $\theta\subset \Omega_{\lambda,0}$.
Note that $\theta$ has size $\lambda^{-1}R^{-\frac12}\times R^{-\frac14}$, the wave packet adapted to $T\in \mathbb{T}(\bar{\theta})$ is given by
\[\psi_{T}(x):= \vert T \vert^{-1}e^{2\pi i \xi_{\theta}\cdot x}\phi_{T}(x),\]
where $\xi_{\theta}$ denotes the centre of $\bar{\theta}$.

\begin{lemma}[Bourgain\cite{Bo91}, Tao\cite{Tao}]\label{lem:hickmanvik}
Let $f$ be a smooth function on $\mathbb{R}^3$. For each given slab $\bar{\theta}:\lambda^{-1}R^{-\frac12}\times R^{-\frac14}\times R^{-1}$, there exists a decomposition
\[f_{\bar{\theta}}(x)=\sum_{T\in \mathbb{T}(\bar{\theta})}f_{T}\psi_{T}(x),\]
where the constants $f_{T}$ satisfy
\begin{equation}\label{equ:wavepakhv}
  \Big(\sum_{T\in \mathbb{T}(\bar{\theta)}}\vert f_{T} \vert^2\Big)^{\frac12}\lesssim \|\hat{f}_{\bar{\theta}}\|_{L^2_{\rm avg}(\bar{\theta})},
\end{equation}
where
\[\|f\|_{L^2_{\rm avg}(E)}:= \Big(\frac{1}{\vert E \vert}\int_{E}\vert f(x) \vert^2dx\Big)^{\frac12}.\]
\end{lemma}

\subsection{Square function estimates associated with $\Omega_{\lambda,0}$}\label{subs:square}

To estimate $\Vert {\mathcal E}g_{\Omega_{\lambda,0}}\Vert_{L^p(B_R)}$, we need the following known lemma.
\begin{lemma}[C\'{o}rdoba\cite{Co} and Fefferman\cite{F70}]\label{lem:corfef}
  For  $2\leq p \leq4$, there holds
  \begin{equation}\label{equ:corfef}
    \|f\|_{L^p(\R^2)}\lesssim \Big\|\big(\sum_{\alpha}\vert f_{\alpha}\vert^2\big)^{\frac12}\Big\|_{L^p(\R^2)},
  \;\;  f_{\alpha}:=\mathcal F^{-1}(\widehat{f}\chi_{\alpha}), \end{equation}
where $\alpha$ denotes $R^{-\frac12}\times R^{-1}$-slab and ${\rm supp}\;\hat{f}\subset \mathcal{N}_{{R^{-1}}}(P^1)$. Here $P^1$ denotes the standard parabola in $\mathbb{R}^2$ and $\mathcal{N}_{{R^{-1}}}(P^1)$ denotes $R^{-1}$-neighborhood of $P^1$.
\end{lemma}

Using  Lemma \ref{lem:corfef}, we are able to prove the following result.

\begin{proposition}\label{prop:3.2}
For $2\leq p\leq 4$, it holds
\begin{equation}\label{equiform}
\|F\|_{L^p(\mathbb{R}^3)}\lesssim \Big\|\big( \sum_{\vartheta\subset \Sigma_{\lambda}}\vert F_{\vartheta}\vert ^2\big)^{\frac12} \Big\|_{L^p(\mathbb{R}^3)},
\end{equation}
where $\Sigma_{\lambda}:=\{(\xi_1,\xi_2,\xi_3): (\xi_1,\xi_2)\in \Omega_{\lambda,0}, \xi_3=\xi^4_1+\xi^4_2\}$,
${\rm supp}\;\hat{F}\subset\mathcal{N}_{R^{-1}}(\Sigma_{\lambda})$, $\vartheta$ is $\lambda^{-1}R^{-\frac12}\times R^{-\frac14}\times R^{-1}$-slab, and $F_{\vartheta}:=\mathcal F^{-1}(\widehat{F}\chi_{\alpha})$.
\end{proposition}

To prove Proposition \ref{prop:3.2}, it is sufficient to show the following lemma
 by freezing the $x_2$ variable.
\begin{lemma}\label{lem:gamwideg}
Let $\Gamma_{\lambda}=\{(t,t^4): t\in [\lambda,2\lambda]\}$ and $2\leq p\leq 4$. It holds
\begin{equation}\label{equ:gamgthet}
  \|G\|_{L^p(\mathbb{R}^2)}\lesssim \Big\|\big( \sum_{\beta}\vert G_{\beta}\vert ^2\big)^{\frac12} \Big\|_{L^p(\mathbb{R}^2)},
\end{equation}
where ${\rm supp}\;\hat{G}\subset\mathcal{N}_{R^{-1}}(\Gamma_{\lambda})$, $\beta$ is $\lambda^{-1}R^{-\frac12}\times R^{-1}$-slab, $G_{\beta}:=\mathcal F^{-1}(\widehat{G}\chi_{\beta})$
\end{lemma}

With Lemma \ref{lem:gamwideg} in hand, we prove Proposition \ref{prop:3.2}. We denote $F(\cdot,x_2,\cdot)$ by $G$.  By the arguments in \cite{Guth18}, it is easy to see that ${\rm supp}\;\hat{G}$ is contained in the projection of ${\rm supp}\;\hat{F}$ on the plane $\xi_2=0$, that is, ${\rm supp}\;\hat{G}\subset \mathcal{N}_{R^{-1}}(\Gamma_{\lambda})$.  Integrating the both sides of \eqref{equ:gamgthet} with respect to $x_2$-variable from $-\infty$ to $\infty$, we obtain \eqref{equiform}, as required.

\begin{proof}[{\bf Proof of Lemma \ref{lem:gamwideg}:}]
  It is sufficient to prove
  \[\Big\|\int^{R^{-1}}_0\mathcal E_{[\lambda,2\lambda]}h(y,s)ds\Big\|_{L^p(\mathbb{R}^2)}\lesssim \Big\|\Big( \sum_{I:\lambda^{-1}R^{-\frac12}-interval}\Big| \int^{R^{-1}}_0\mathcal E_{I}h(y,s)ds\Big| ^2\Big)^{\frac12} \Big\|_{L^p(\mathbb{R}^2)},\]
  where
  \[\mathcal E_{[\lambda,2\lambda]}h(y,s):=\int_{[\lambda,2\lambda]}h(t)e[y_1t+y_2(t^4+s)]dt.\]
  Using variable substitutions, $t=\lambda u+\lambda, s=\lambda^4 v$,  denoted by $\mathcal{L}$, we have
  \[\Big\vert\int^{R^{-1}}_0(\mathcal E_{[\lambda,2\lambda]}h)(y,s)ds\Big\vert=:\Big\vert\lambda^{-1}\int^{\lambda^{-4}R^{-1}}_0(\mathcal E^{\rm Parp}_{[0,1]}\tilde{h})(\tilde{y},v)dv\Big\vert,\]
  where
  \[\tilde{h}(u,v)=\lambda h(\lambda u+\lambda, \lambda^4 v), \;\;\; \tilde{y}=(\lambda y_1+4\lambda^4y_2,\lambda^4y_2),\]
  and $\mathcal E^{\rm Parp}_{[0,1]}\tilde{h}(\tilde{y},v)$ denotes the Fourier extension operator associated with the non-degenerate curve: \[\gamma(u)=\big(u,6u^2+4u^3+u^4+v\big), \;\;\; u\in [0,1], \;\;\; v\in [0,\lambda^{-4}R^{-1}].\]
  Now we apply Lemma \ref{lem:corfef} to $\int^{\lambda^{-4}R^{-1}}_0\mathcal{E}^{\rm Parp}_{[0,1]}\tilde{h}(\tilde{y},v)dv$. Since
  \[\Big\Vert \int^{\lambda^{-4}R^{-1}}_0\mathcal E_{[\lambda,2\lambda]}h(y,s)ds \Big\Vert^p_{L^p(\mathbb{R}^2)}=\lambda^{-1}\Big\Vert \int^{\lambda^{-4}R^{-1}}_0\mathcal{E}^{\rm Parp}_{[0,1]}\tilde{h}(\tilde{y},v)dv\Big\Vert^p_{L^p(\mathbb{R}^2)},\]
  from Lemma \ref{lem:corfef}, the right hand side of last equality is controlled by
  \[\lambda^{-1}\Big\|\Big(\sum_{\tilde{I}:\lambda^{-2}R^{-\frac12}-{\rm interval}}\big\vert\ \int^{\lambda^{-4}R^{-1}}_0\mathcal{E}^{\rm Parp}_{\tilde{I}}\tilde{h}(\tilde{y},v)dv\big\vert^2\Big)^{\frac12}\Big\|^p_{L^p(\mathbb{R}^2)}.\]
  Taking the inverse of $\mathcal{L}$, we get the desired estimate \eqref{equ:gamgthet}.
 \end{proof}

\subsection{An argument using 2D Kakeya-type estimate}

In the discussion below, our strategy is to interpolate a $L^4$ estimate with a trivial $L^2$ bound. Fix $R\gg 1$ and fixed bump function $\varphi\in C^{\infty}_c(\mathbb{R}^3)$ with ${\rm supp}\; \varphi \subset B(0,1)$ and $\vert  \check{\varphi}(x) \vert\geq 1$ for all $x\in B(0,1)$.
Defining $f:= \check{\varphi}_{R^{-1}} {\mathcal E}g_{\Omega_{\lambda,0}}$ where $\varphi_{R^{-1}}(\xi):= R^3\varphi(R\xi)$, it follows
\begin{equation}\label{eq-add-0}
\|{\mathcal E}g_{\Omega_{\lambda,0}}\|_{L^p(B_R)}\lesssim
\|\check{\varphi}_{R^{-1}}{\mathcal E}g_{\Omega_{\lambda,0}}\|_{L^p(\mathbb{R}^3)}=\|f\|_{L^p(\mathbb{R}^3)}.
\end{equation}

By Proposition \ref{prop:3.2} with $p=4$, we have
\begin{equation}\label{equ:egome0est123}
  \|f\|_{L^4(\mathbb{R}^3)}\lesssim \Big\|\big( \sum_{\bar{\theta}}\vert f_{\bar{\theta}}\vert ^2\big)^{\frac12} \Big\|_{L^4(\mathbb{R}^3)},
\end{equation}

Applying the wave packet decomposition to the function $f_{\bar{\theta}}$
 \[f_{\bar{\theta}}(x)=\sum_{T\in \mathbb{T}(\bar{\theta})}f_{T}\psi_{T}(x),\]
we get by Lemma \ref{lem:hickmanvik}
\begin{equation}\label{equ:wavepakhv123}
  \Big(\sum_{T\in \mathbb{T}(\bar{\theta)}}\vert f_{T} \vert^2\Big)^{\frac12}\lesssim \|\hat{f}_{\bar{\theta}}\|_{L^2_{\rm avg}(\bar{\theta})}.
\end{equation}
This inequality together with \eqref{eq-add-0} implies that
\begin{equation}\label{equ:kakeyamax}
 \|{\mathcal E}g_{\Omega_{\lambda,0}}\|_{L^4(B_R)}\lesssim \lambda^{-1}R^{-\frac74}\Big\| \Big[\sum_{\bar{\theta}}\Big(\sum_{T\in \mathbb{T}(\bar{\theta})}\vert f_{T}\vert\cdot \vert\psi_{T}\vert\Big)^2\Big]^{\frac12}\Big\|_{L^4(\mathbb{R}^3)}
\end{equation}
Let $\chi_{T,\ell}$ denote the characteristic  function of $a_T([-\frac14,\frac14]^n+\frac{\ell}2)$
for each $\ell\in\mathbb Z^n$ such that the $\{\chi_{T,\ell}\}_{\ell\in \mathbb Z^n}$ form
a rough partition of unity of $\R^n$. Note that
\begin{equation*}
  \vert \psi_{T}(x) \vert = \sum_{\ell\in \mathbb{Z}^3}\vert \psi_{T}(x) \vert\chi_{T,\ell}(x)\lesssim \sum_{\ell\in \mathbb{Z}^3}\frac{\chi_{T,\ell}(x)}{(1+\vert \ell \vert)^4},
\end{equation*}
we have
\begin{align*}
\|{\mathcal E}g_{\Omega_{\lambda,0}}\|_{L^4(B_R)} \lesssim& \lambda^{-1}R^{-\frac74}\sum_{\ell\in \mathbb{Z}^3}(1+\vert \ell \vert)^{-4}\Big\|\Big[\sum_{\bar{\theta}}\Big(\sum_{T\in \mathbb{T}(\bar{\theta})}\vert f_{T}\vert \chi_{T,\ell}\Big)^2\Big]^{\frac12} \Big\|_{L^4(\mathbb{R}^3)}\\
\lesssim &\lambda^{-1}R^{-\frac74}\sum_{\ell\in \mathbb{Z}^3}(1+\vert \ell \vert)^{-4}\Big\|\sum_{\bar{\theta}}\sum_{T\in \mathbb{T}(\bar{\theta})}\vert f_{T}\vert^2\chi_{T,\ell} \Big\|^{\frac12}_{L^2(\mathbb{R}^3)},
\end{align*}
where we have used the fact that the supports of the $\chi_{T,\ell}$ are essentially disjoint as $T$ varies over $\mathbb T(\bar \theta)$.
Assume that $\|\hat{f}\|_{L^{\infty}(\mathcal{N}_{R^{-1}}(\Sigma_{\lambda}))}=1$, we
have by \eqref{equ:wavepakhv123}
$$\Big(\sum_{T\in \mathbb{T}(\bar{\theta})}\vert f_{T} \vert^2\Big)\lesssim \|\hat{f}\|^2_{L^2_{\rm avg}(\bar{\theta})}\leq 1$$
for each $\bar{\theta}$. There exists a sequences $\{c_{T}\}_{T\in \mathbb{T}(\bar{\theta})}$ of non-negative real numbers with
\begin{equation}\label{eq-add-2}\sum_{T\in \mathbb{T}(\bar{\theta})}c_{T}=1,\end{equation}
such that
\[\Big\|\sum_{\bar{\theta}}\sum_{T\in \mathbb{T}(\bar{\theta})}\vert f_{T}\vert^2\chi_{T,\ell} \Big\|^{\frac12}_{L^2(\mathbb{R}^3)}\lesssim \Big\|\sum_{\bar{\theta}}\sum_{T\in \mathbb{T}(\bar{\theta})}c_{T}\chi_{T,\ell} \Big\|^{\frac12}_{L^2(\mathbb{R}^3)}.\]
For each $\bar{\theta}$, let $\{S_{T}: T\in \mathbb{T}(\bar{\theta})\}$ be a collection of pairwise disjoint subsets of $[0,1]$ with $\vert S_{T}\vert = c_{T}$. Using H\"{o}lder's inequality and the condition \eqref{eq-add-2}, we get
\begin{align*}
\Big\|\sum_{\bar{\theta}}\sum_{T\in \mathbb{T}(\bar{\theta})}c_{T}\chi_{T,\ell}\Big\|^{2}_{L^2(\mathbb{R}^3)}=&
\int_{\mathbb{R}^3}\Big(\int^1_0 \sum_{\bar{\theta}}\sum_{T\in \mathbb{T}(\bar{\theta})}\chi_{T}(x)\chi_{S_{T}}(t)dt\Big)^2dx \\ \leq &
\int^1_0 \Big[\int_{\mathbb{R}^3}\Big(\sum_{\bar{\theta}}\sum_{T\in \mathbb{T}(\bar{\theta})}\chi_{T}(x)\chi_{S_{T}}(t)\Big)^2dx\Big]dt
\end{align*}
By the definition of $S_T$, for any fixed $\bar{\theta}$ and every $t\in [0,1]$, there exists at most one $T_{\bar{\theta}}\in \mathbb{T}(\bar{\theta})$ such that $\chi_{S_{T_{\bar{\theta}}}}(t)\neq 0$. Hence, the right hand side of the last inequality is controlled by
\[\int^1_0\big\|\sum_{\bar{\theta}}\chi_{T_{\bar{\theta}}}\big\|^2_{L^2(\mathbb{R}^3)}dt.\]

\begin{proposition}\label{prop:2dkinq}
There holds
\begin{equation}\label{equ:theatest}
 \big \|\sum_{\bar{\theta}}\chi_{T_{\bar{\theta}}}\big\|^2_{L^2(\mathbb{R}^3)}\lesssim \lambda^{\varepsilon}\lambda R^{\varepsilon}R^{\frac94}.
\end{equation}

\end{proposition}

\begin{proof}
  We will apply C\'ordoba's Kakeya-type argument in $\R^2$ to estimate the left hand side of \eqref{equ:theatest}.
   Suppose that
\[{\rm Angle}({\rm dir}(T_{\bar{\theta}_1}),{\rm dir}(T_{\bar{\theta}_2}))\approx jR^{-\frac12},\;\;
\;1\leq j \lesssim R^{\frac12},\;\;\; j\in \mathbb{N}. \]
Since $T_{\bar{\theta}}$  is  $R \times \lambda R^{\frac12}\times R^{\frac14}$-rectangle, we have
\[\mid T_{\bar{\theta}_1}\cap T_{\bar{\theta}_2}\mid \lesssim \tfrac{\lambda R^{7/4}}{j}.\]
It follows that
\begin{align*}
   \Big\|\sum_{\bar{\theta}}\chi_{T_{\bar{\theta}}}\Big\|^2_{L^2(\mathbb{R}^3)}=&\int_{\R^3} \sum_{\bar{\theta}_1,\bar{\theta}_2}\chi_{T_{\bar{\theta}_1}}\chi_{T_{\bar{\theta}_2}}dx
   =\sum_{\bar{\theta}_1,\bar{\theta}_2}\big(\big| T_{\bar{\theta}_1}\cap T_{\bar{\theta}_2}\big|\big)\\
   =&\sum_{\bar{\theta}_1}\Big(\sum_{\bar{\theta}_2}\big| T_{\bar{\theta}_1}\cap T_{\bar{\theta}_2}\big|\Big)
   \lesssim \sum_{\bar{\theta}_1}\sum^{[R^{\frac12}]}_{j=1}\frac{\lambda R^{\frac74}}{j}\\
   \approx& \lambda R^{\frac74}\log R\cdot \mathrm{Card}(T_{\bar{\theta}_1}) \lesssim\lambda R^{9/4}\log R.
\end{align*}
This completes the proof of Proposition \ref{prop:2dkinq}.

\end{proof}

From the above discussion, we obtain
\begin{equation}\label{equ:egbrest}
  \|\mathcal{E}g_{\Omega_{\lambda,0}}\|_{L^4(B_R)}\lesssim_{\lambda}R^{\varepsilon}R^{-\frac{3}{16}}\|g\|_{L^{\infty}(\Sigma)}.
\end{equation}
Interpolating it with a trivial $L^2$-bound
\begin{equation}\label{equ:l2estboud}
  \|\mathcal{E}g_{\Omega_{\lambda,0}}\|_{L^2(B_R)}\lesssim_{\lambda}R^{\frac{1}{2}}\|g_{\Omega_{\lambda,0}}\|_{L^2}
  \lesssim_{\lambda}R^{\frac38}\|g\|_{L^{\infty}},
\end{equation}
we obtain
\begin{equation}\label{equ:omega0rest}
  \|\mathcal{E}g_{\Omega_{\lambda,0}}\|_{L^p(B_R)}\lesssim C(K)R^{\frac{9}{4p}-\frac{3}{4}+\varepsilon}\|g\|_{L^{\infty}(\Sigma)},
\end{equation}
where $C(K)$ is a fixed positive power of $K$ and we have used the fact  $K^{-\frac14}\leq \lambda \leq \frac{1}{2}$.
This inequality  yields  the claim  \eqref{equ:claimlambda0}.
And so, we conclude the proof of Theorem \ref{thm:main}.




\subsection*{Acknowledgements}   C.~Miao and J.~Zheng were supported by NSF of China No.11831004. The authors are also grateful to Professor Shaoming Guo for his valuable discussion.

\bigskip

\begin{center}

\end{center}

\end{document}